\documentclass[12pt,reqno,a4paper]{amsart}

\usepackage{hyperref}

\usepackage{comment}

\usepackage{amsmath,amsfonts,amsthm,mathrsfs,amssymb,cite}
\usepackage[usenames]{color}

\newtheorem{thm}{Theorem}[section]

\newtheorem{lem}{Lemma}[section]

\theoremstyle{definition}

\theoremstyle{remark}

\newtheorem{rem}{Remark}[section]
\numberwithin{equation}{section}

\title[Uniqueness in Thermo and Photoacoustic Tomography]{Determining both sound speed and internal source in thermo- and photo-acoustic tomography}

\author{Hongyu Liu}
\address{Department of Mathematics, Hong Kong Baptist University,Kowloon Tong, Hong Kong SAR.\vspace*{-3mm}}
\address{\vspace*{-2mm}and}
\address{HKBU Institute of Research and Continuing Education, Virtual University Park, Shenzhen, P. R. China.}
\email{hongyu.liuip@gmail.com}

\author{Gunther Uhlmann}
\address{Department of Mathematics, University of Washington, Seattle, WA 98195, USA. \vspace*{-3mm}}
\address{\vspace*{-2mm} and}
\address{Institute for Advanced Study, Hong Kong University of Science and Technology, Hong Kong SAR.}
\email{gunther@math.washington.edu}

\begin{document}

\begin{abstract}
This paper concerns thermoacoustic tomography and photoacoustic tomography, two couple-physics imaging modalities that attempt to combine the high resolution of ultrasound and the high contrast capabilities of electromagnetic waves. We give sufficient conditions to recover both the sound speed of the medium being probed and the source.
\end{abstract}




\maketitle


\section{Introduction}

In this paper we consider thermoacoustic tomography (TAT) and photoacoustic tomography (PAT), two coupled-physics modalities that attempt to combine the high resolution of ultrasound and the high contrast capabilities of electromagnetic waves. In the latter a laser pulse is used to probe a medium and in the former a lower frequency wave is used. As a result the medium is heated and produces an elastic expansion that in turns produces a sound wave that is measured outside the medium. This is used to find optical properties of the medium in the case of PAT or electromagnetic properties in the case of TAT. These two imaging modalities and other coupled-physics imaging techniques have the potential of significant clinical and biological applications \cite{KRK,SU4,W,XW}.


We next present the mathematical formulation of the PAT and TAT problem described above. Let $\Omega\subset\mathbb{R}^3$ be a compact set such that $\mathbb{R}^3\backslash \Omega$ is connected. Let $c(x)\in L^\infty(\mathbb{R}^3)$ be such that $supp(1-c)\subset \Omega$, and $c(x)\geq c_0$ for $x\in \Omega$, where $c_0$ is a positive constant. Consider the following Cauchy problem of the wave equation for $u(x,t)$, $(x, t)\in \mathbb{R}^3\times\mathbb{R}_+$,
\begin{equation}\label{eq:1}
\begin{cases}
& \displaystyle{\frac{1}{c^2(x)}\partial_t^2 u(x, t)-\Delta u(x, t)=0 \quad\ \mbox{in\ \ $\mathbb{R}^3\times \mathbb{R}_+,$}}\smallskip \\
& u(x,0)=f(x),\ \ \ \partial_t u(x,0)=0\quad \mbox{for \ $x\in\mathbb{R}^3$},
\end{cases}
\end{equation}
where $f(x)\in L^\infty(\mathbb{R}^3)$ with $supp(f)\subset \Omega$. $c(x)$ represents the wave speed in the body $\Omega$ and $f(x)$ represents an internal source produced by the tissue volume expansion. The wave measurements are given by
\begin{equation}\label{eq:2}
\Lambda_{f,c}(x,t)=u(x, t),\quad (x, t)\in {\partial \Omega\times\mathbb{R}_+}. 
\end{equation}
The inverse problem is to recover $(f, c)$ by knowledge of $\Lambda_{f,c}(x,t)$. In what follows, for simplicity, $(f, c)$ shall be referred to as a TAT/PAT configuration. 

The uniqueness issue for the TAT/PAT problem has gained extensive attention in the literature. However, most of the available results intend to recover either $f$ or $c$, by assuming that the other one is known. For the recovery of $f$ by assuming that $c$ is known, we refer to \cite{KK} and the references therein for the constant sound speed case; and \cite{SU2} for the variable sound speed case. In \cite{SU}, the recovery of a smooth sound speed $c$ by assuming that the source $f$ is known was considered. Uniqueness was established in \cite{SU} with one measurement under the geometric assumption that the domain is foliated by strictly convex hypersurfaces with respect to the Riemannian metric $g=\frac{1}{c^2} dx^2$, where $dx^2$ denotes the Euclidean metric. In \cite{SU3}, the instability of the linearized problem in recovering both the internal source and the sound speed was established. In \cite{OU}, stability estimate was derived in recovering $f$ if there is a modeling error of the variable sound speed. In \cite{FH}, an interesting connection was observed between the TAT/PAT problem and the interior transmission eigenvalue problem. 

In this work, we show that in certain practically interesting scenarios, one can recover both the sound speed and the internal source by a single measurement. Our argument relies on the temporal Fourier transform, converting the TAT/PAT problem into the frequency domain. Then by using the low frequency asymptotics, we derive two internal identities involving the unknowns. The main uniqueness that we could derive from the two integral identities is that if the sound speed is constant and the internal source function is independent of one variable, then both of them can be recovered. We also point out some other uniqueness results in our study. Moreover, our argument requires the least regularity assumptions on the parameters as well as the domain. The rest of the paper is organized as follows. In Section 2, we present the main results of the current study. Sections 3 is devoted to the proofs.

\section{Statement of the main results}

We first introduce the admissible class of TAT/PAT configurations $(f, c)$ for our study. Let $(f, c)$, $\Omega$ and $u(x, t)$ be given in \eqref{eq:1}. Our argument relies essentially on the {\it temporal Fourier transform} of the function $u(x,t)$ defined by
\begin{equation}\label{eq:t1}
\hat{u}(x, k)=\frac{1}{2\pi}\int_0^\infty u(x, t) e^{ik t}\ dt,\quad (x, k)\in\mathbb{R}^3\times \mathbb{R}_+. 
\end{equation}
A TAT/PAT configuration $(f, c)$ is said to be admissible if \eqref{eq:t1} defines an $H_{loc}^1(\mathbb{R}^3)$ distribution for $k\in (0,\epsilon_0)$ with a certain $\epsilon_0\in\mathbb{R}_+$. It is remarked that a sufficient condition to guarantee the admissibility of a TAT/PAT configuration $(f, c)$ is that $c$ is a non-trapping sound speed; see \cite{FH,SU} and the references therein for more relevant discussion. 

In what follows, we let
\[
\Phi(x)=\frac{e^{ik|x|}}{4\pi |x|}\quad\mbox{for\ \ $|x|\neq 0$}.
\]
$\Phi$ is the fundamental solution to $-\Delta-k^2$. For the subsequent use, we also set
\begin{equation}\label{eq:tt4}
\Phi_0(x)=\frac{1}{4\pi |x|}\quad\mbox{for\ \ $|x|\neq 0$}.
\end{equation}
Let $B_R:=B_R(0)$ denote a central ball of radius $R\in\mathbb{R}_+$ containing $\Omega$. Moreover, for a function $\psi\in L^\infty(\mathbb{R}^3)$ supported in $B_R$, we denote
\begin{equation}\label{eq:linverse}
\Delta^{-1}(\psi)(x):=\int_{\mathbb{R}^3}\Phi_0(x-y)\cdot \psi(y)\ dy,\quad x\in\mathbb{R}^3. 
\end{equation}

We shall prove that

\begin{thm}\label{thm:1}
Let $(f_j, c_j)$, $j=1,2$, be two admissibleTAT/PAT configurations such that 
\begin{equation}\label{eq:equal} 
\Lambda_{f_1,c_1}(x,t)=\Lambda_{f_2,c_2}(x,t),\quad (x,t)\in\partial \Omega\times\mathbb{R}_+.
\end{equation}
Then one has
\begin{equation}\label{eq:orth1}
\int_{\mathbb{R}^3}(c_1^{-2}f_1-c_2^{-2}f_2)(x)\cdot v(x)\ dx=0,
\end{equation}
where $v(x)$ is an arbitrary harmonic function, and
\begin{equation}\label{eq:orth2}
\begin{split}
& \int_{\mathbb{R}^3} (1-c_1^{-2})(y)\cdot \Delta^{-1}(c_1^{-2}f_1)(y)\cdot \Phi_0(x-y)\ dy\\
&\qquad +\frac{1}{8\pi}\int_{\mathbb{R}^3} (c_1^{-2}f_1)(y)\cdot |x-y|\ dy\\
=& \int_{\mathbb{R}^3} (1-c_2^{-2})(y)\cdot \Delta^{-1}(c_2^{-2}f_2)(y)\cdot \Phi_0(x-y)\ dy\\
&\qquad +\frac{1}{8\pi}\int_{\mathbb{R}^3} (c_2^{-2}f_2)(y)\cdot |x-y|\ dy,
\end{split}
\end{equation}
which holds for any $x\in\partial B_R$.
\end{thm}


%
%

It is noted that in \eqref{eq:orth1} and \eqref{eq:orth2} the integrations can be taken over $\Omega$ or $B_R$ instead of $\mathbb{R}^3$ since both $(f_1, c_1)$ and $(f_2, c_2)$ are supported in $\Omega$. In the sequel, we let $x=(x_j)_{j=1}^3\in\mathbb{R}^3$. Using Theorem~\ref{thm:1}, one can show that

\begin{thm}\label{thm:2}
Let $(f_j, c_j)$, $j=1,2$, be two admissibleTAT/PAT configurations. Then one has that
\begin{equation}\label{eq:2eq}
c_1^{-2}(x)f_1(x)=c_2^{-2}(x) f_2(x)\quad\mbox{for a.e.}\ x\in\Omega,
\end{equation}
if either of the following two conditions is satisfied\smallskip
\begin{enumerate}
\item[$i).$]~$\varphi(x):=c_1^{-2}(x)f_1(x)-c_2^{-2}(x) f_2(x)$ is a harmonic function in $\Omega$;\smallskip

\item[$ii).$]~$\varphi(x)$ from item i) is independent of one variable, say $x_j$, $1\leq j\leq 3$. 
\end{enumerate}

\end{thm}

\begin{rem}
By Theorem~\ref{thm:2}, one readily sees that if an admissible TAT/PAT configuration $(f, c)$ is such that either $c^{-2}f$ is a harmonic function in $\Omega$, or $c^{-2}f$ is a function independent of one variable $x_j$, $1\leq j\leq 3$, then $c^{-2} f$ is uniquely determined by $\Lambda_{f, c}(x, t)$ for $(x, t)\in \partial \Omega\times\mathbb{R}_+$. Using this result, one can further infer that if the sound speed $c$ is known, then the internal source $f$ is uniquely determined; whereas if $f$ is known in advance, then the sound speed $c|_{supp(f)}$ is uniquely determined. 
\end{rem}

Next, we present a practical scenario in uniquely determining both the sound speed and the internal source. 

\begin{thm}\label{thm:25}
Let $(f, c)$ be an admissible TAT/PAT configuration supported in $\Omega$ such that $c$ is constant, and $f$ is independent of one variable $x_j$, $1\leq j\leq 3$. Furthermore, we assume that 
\begin{equation}\label{eq:bp25}
f(x)\geq 0\quad\mbox{for a.e.}\ x\in\Omega\quad\mbox{and}\quad \int_{\Omega} f(x)\ dx>0.
\end{equation}
Then both $f$ and $c$ are uniquely determined by $\Lambda_{f, c}(x, t)$ for $(x, t)\in\partial \Omega\times\mathbb{R}_+$.
\end{thm}

\begin{rem}\label{rem:22}
Assumption~\ref{eq:bp25} on the internal source $f$ means that there exists an open portion of $\Omega$ on which $f$ is strictly positive. Since $f$ is generated by the absorbed energy, it is a practically reasonable condition. 
\end{rem}

By generalizing Theorem~\ref{thm:25}, we can further prove a bit more general unique determination result as follows. 

\begin{thm}\label{thm:3}
Let $(f, c)$ be an admissible TAT/PAT configuration such that 
\begin{equation}\label{eq:bp}
c^{-2}(x)=c_b^{-2}(x)+\gamma^{-2}\chi_{\Sigma},\quad x\in\Omega,
\end{equation}
where $c_b$ is a positive background sound speed which is supposed to be known in advance, and $\gamma$ is a positive constant denoting an anomalous inclusion supported in $\Sigma\subset \Omega$. Furthermore, we assume that both $c_b$ and $f$ are independent of the variable $x_j$, $1\leq j\leq 3$ and
\begin{equation}\label{eq:bp2}
\int_{\Sigma}\Delta^{-1}(c^{-2}f)(x)\cdot v(x)\ dx\neq 0
\end{equation}
for a harmonic function $v$. Then both $f$ and $\gamma$ are uniquely determined by $\Lambda_{f, c}(x, t)$ for $(x, t)\in\partial \Omega\times\mathbb{R}_+$.
\end{thm}

From our subsequent argument, one can see that if $c_b$ and $\Sigma$ in Theorem~\ref{thm:3} are taken to be $0$ and $\Omega$, and moreover $v\equiv 1$, then Theorem~\ref{thm:25} is a particular case of Theorem~\ref{thm:3}. 

\section{Proofs}

It is straightforward to show that $\hat{u}(x, k)$ satisfies
\begin{equation}\label{eq:t2}
\Delta \hat u(x, k)+\frac{k^2}{c^2(x)} \hat u(x, k)=\frac{ik}{2\pi}\frac{1}{c^2(x)} f(x),\quad (x, k)\in\mathbb{R}^3\times (0, \epsilon_0).
\end{equation}
Moreover, in order to ensure the well-posedness of the transformed equation \eqref{eq:t2}, one imposes the classical Sommerfeld radiation condition 
\begin{equation}\label{eq:t3}
\lim_{|x|\rightarrow+\infty}|x| \left(\frac{\partial \hat u(x, k)}{\partial |x|}-i k\hat u(x,k) \right)=0. 
\end{equation}
Under the temporal Fourier transform, the TAT/PAT measurement data become
\begin{equation}\label{eq:t4}
\widehat{\Lambda}(x, k)=\hat u(x, k)|_{\partial \Omega\times (0, \epsilon_0)}. 
\end{equation}

The following lemma shall be needed. 

\begin{lem}\label{lem:1}
Let $\hat u(x, k) \in H_{loc}^1(\mathbb{R}^3)$ be the solution to \eqref{eq:t2}-\eqref{eq:t3}. Then $\hat u(x, k)$ is uniquely given by the following integral equation
\begin{equation}\label{eq:t5}
\begin{split}
\hat u(x, k)=& -k^2\int_{\mathbb{R}^3}\left[1-c^{-2}(y)\right]\cdot \hat{u}(y, k)\cdot\Phi(x-y)\ dy\\
& -\frac{ik}{2\pi}\int_{\mathbb{R}^3} c^{-2}(y)\cdot f(y)\cdot \Phi(x-y)\ dy,\quad x\in\mathbb{R}^3. 
\end{split}
\end{equation}
Moreover, as $k\rightarrow +0$, we have
\begin{equation}\label{eq:t6}
\hat u(x, k)= -\frac{ik}{2\pi}\int_{\mathbb{R}^3} c^{-2}(y)\cdot f(y)\cdot \Phi_0(x-y)\ dy+\mathcal{O}(k^2),\quad x\in B_R. 
\end{equation}
\end{lem}

\begin{proof}
By the elliptic regularity (see, e.g, \cite{McL}), we known that $u\in H_{loc}^2(\mathbb{R}^3)$. The integral equation \eqref{eq:t5} is of Lippmann-Schwinger type, and it is uniquely solvable for $u\in C(B_R)$ (cf. \cite{CK}).  By using the fact that
\begin{equation}\label{eq:estimate1}
\Phi(x)=\Phi_0(x)+\mathcal{O}(k)\quad \mbox{as} \ \ k\rightarrow +0,
\end{equation}
uniformly for $x\in B_R$, one can verify \eqref{eq:t6} by straightforward asymptotic estimates. 

The proof is complete. 
\end{proof}

We are in a position to present the proof of Theorem~\ref{thm:1}.

\begin{proof}[Proof of Theorem~\ref{thm:1}]
Let $\hat u_1(x, k)$ and $\hat u_2(x, k)$ denote the acoustic wave fields, respectively, corresponding to $(f_1, c_1)$ and $(f_2, c_2)$. By \eqref{eq:equal} and \eqref{eq:t4}, one clearly has that $\hat{u}_1(x, k)=\hat u_2(x, k)$ for $(x, k)\in \partial \Omega\times (0,\epsilon_0)$. Since both $\hat u_1$ and $\hat u_2$ satisfy the same equation $(\Delta+k^2)u=0$ in $\mathbb{R}^3\backslash\overline{\Omega}$, one readily has that
\begin{equation}\label{eq:pp1}
\hat u_1(x, k)=\hat u_2(x, k)\quad\mbox{for}\ \ (x, k)\in (\mathbb{R}^3\backslash \Omega)\times (0, \epsilon_0), 
\end{equation}
and therefore
\begin{equation}\label{eq:pp2}
\hat u_1(x, k)|_{\partial B_R\times (0, \epsilon_0)}=\hat u_2(x, k)|_{\partial B_R\times (0, \epsilon_0)}. 
\end{equation}
Hence, by \eqref{eq:t5} and \eqref{eq:pp2}, we see that for $x\in\partial B_R$,
\begin{equation}\label{eq:pp6}
\begin{split}
&-k\int_{B_R}\left[1-c_1^{-2}(y)\right]\cdot \hat{u}_1(y, k)\cdot\Phi(x-y)\ dy\\
& -\frac{i}{2\pi}\int_{B_R} c_1^{-2}(y)\cdot f_1(y)\cdot \Phi(x-y)\ dy\\
=&-k\int_{B_R}\left[1-c_2^{-2}(y)\right]\cdot \hat{u}_2(y, k)\cdot\Phi(x-y)\ dy\\
& -\frac{i}{2\pi}\int_{B_R} c_2^{-2}(y)\cdot f_2(y)\cdot \Phi(x-y)\ dy.
\end{split}
\end{equation}

Next, by Lemma~\ref{lem:1} one has
\begin{equation}\label{eq:pp3}
\hat u_j(x, k)=-\frac{ik}{2\pi}\int_{B_R} c_j^{-2}(y)\cdot f_j(y)\cdot \Phi_0(x-y)\ dy+\mathcal{O}(k^2),\ \ j=1,2,
\end{equation}
for $k\in\mathbb{R}_+$ sufficiently small. Plugging \eqref{eq:pp3} into \eqref{eq:pp6} and using the series expansion for $e^{ik|x-y|}$ of $\Phi(x-y)$, one can further show by straightforward calculations (though a bit tedious) that as $k\rightarrow +0$
\begin{equation}\label{eq:ted1}
\begin{split}
&\frac{i}{2\pi}\int_{B_R} c_1^{-2}(y)\cdot f_1(y)\cdot \Phi_0(x-y)\ dy-\frac{k}{8\pi^2}\int_{B_R} c_1^{-2}(y)\cdot f_1(y)\ dy\\
&-\frac{i}{2\pi}\cdot k^2\int_{B_R} (1-c_1^{-2}(y))\cdot\Delta^{-1}(c_1^{-2}f_1)(y)\cdot\Phi_0(x-y)\ dy\\
&-\frac{i}{16\pi^2}\cdot k^2 \int_{B_R} c_1^{-2}(y)\cdot f_1(y)\cdot |x-y|\ dy+\mathcal{O}(k^3)\\
=&\frac{i}{2\pi}\int_{B_R} c_2^{-2}(y)\cdot f_2(y)\cdot \Phi_0(x-y)\ dy-\frac{k}{8\pi^2}\int_{B_R} c_2^{-2}(y)\cdot f_2(y)\ dy\\
&-\frac{i}{2\pi}\cdot k^2\int_{B_R} (1-c_2^{-2}(y))\cdot\Delta^{-1}(c_2^{-2}f_2)(y)\cdot\Phi_0(x-y)\ dy\\
&-\frac{i}{16\pi^2}\cdot k^2 \int_{B_R} c_2^{-2}(y)\cdot f_2(y)\cdot |x-y|\ dy+\mathcal{O}(k^3),
\end{split}
\end{equation}
which holds for any $x\in\partial B_R$. 

From \eqref{eq:ted1}, one immediately has \eqref{eq:orth2} and for any $x\in\partial B_R$,
\begin{equation}\label{eq:pp7}
\int_{B_R} c_1^{-2}(y)\cdot f_1(y)\cdot\Phi_0(x-y)\ dy=\int_{B_R} c_2^{-2}(y)\cdot f_2(y)\cdot\Phi_0(x-y)\ dy. 
\end{equation}

We proceed to the proof of the orthogonality relation \eqref{eq:orth1}, and first note the following addition formula for $\Phi_0(x-y)$ with $|x|>|y|$ (cf. \cite{CK}),
\begin{equation}\label{eq:pp5}
\frac{1}{4\pi |x-y|}=\sum_{\alpha=0}^\infty\sum_{\beta=-\alpha}^\alpha \frac{1}{2\alpha+1}\frac{|y|^\alpha}{|x|^{\alpha+1}} Y_\alpha^\beta(\hat x) \overline{Y_\alpha^\beta}(\hat y),
\end{equation}
where $\hat x=x/|x|$ and $\hat y=y/|y|$, and $Y_\alpha^\beta(\hat x)$ denotes the spherical harmonics of order $\alpha\in\mathbb{N}\cup\{0\}$, for $\beta=-\alpha,\ldots,\alpha$. By inserting \eqref{eq:pp5} into \eqref{eq:pp7}, together with straightforward calculations, one then has
\begin{equation}\label{eq:pp9}
\begin{split}
& \sum_{\alpha=0}^\infty\sum_{\beta=-\alpha}^\alpha \frac{1}{2\alpha+1}\cdot\frac{1}{R^{\alpha+1}}\cdot Y_\alpha^\beta(\hat x) \int_{B_R} c_1^{-2}(y)\cdot f_1(y)\cdot |y|^{\alpha} \cdot \overline{Y_\alpha^\beta}(\hat y)\ dy\\
=&  \sum_{\alpha=0}^\infty\sum_{\beta=-\alpha}^\alpha \frac{1}{2\alpha+1}\cdot\frac{1}{R^{\alpha+1}}\cdot Y_\alpha^\beta(\hat x) \int_{B_R} c_2^{-2}(y)\cdot f_2(y)\cdot |y|^{\alpha} \cdot \overline{Y_\alpha^\beta}(\hat y)\ dy,
\end{split}
\end{equation}
which holds for any $x\in\partial B_R$. Since $\{Y_\alpha^\beta(\cdot)\}_{\alpha=0,1,\ldots, \beta=-\alpha,\ldots,\alpha}$ is a complete orthonormal basis of $L^2(\mathbb{S}^2)$, where $\mathbb{S}^2$ denotes the unit sphere, one immediately has from \eqref{eq:pp9} that 
\begin{equation}\label{eq:pp10}
\begin{split}
& \int_{B_R} c_1^{-2}(y)\cdot f_1(y)\cdot |y|^\alpha\cdot \overline{Y_n^m}(\hat y)\ dy\\
=&  \int_{B_R} c_2^{-2}(y)\cdot f_2(y)\cdot |y|^\alpha \cdot \overline{Y_n^m}(\hat y)\ dy,
\end{split}
\end{equation}
for $\alpha=0,1,2,\ldots$, and $\beta=-\alpha,\ldots,\alpha$. We note that  $|y|^\alpha Y_\alpha^\beta(\hat y)$ for $\alpha=0,1,2,\ldots$, and $\beta=-\alpha,\ldots,\alpha$ yield all the homogeneous harmonic polynomials. Hence, \eqref{eq:pp10} readily implies \eqref{eq:orth1}. 

The proof is complete.   
\end{proof}

\begin{proof}[Proof of Theorem~\ref{thm:2}]
For case i), if $\varphi(x)$ is a harmonic function in $\Omega$, one can take $v(x)$ to be a harmonic function such that $v(x)=\varphi(x)$ for $x\in\Omega$. Then by \eqref{eq:orth1} in Theorem~\ref{thm:1}, one immediately has that $\varphi\equiv 0$. 

For case ii), without loss of generality, we assume that $\varphi(x)$ is independent of $x_3$ for $x=(x_1,x_2,x_3)\in\mathbb{R}^3$. Let 
\begin{equation}\label{eq:expf}
v(x)=e^{i \zeta\cdot x}\quad x\in\mathbb{R}^3,
\end{equation}
where
\begin{equation}\label{eq:expf2}
\zeta=\xi+i\xi^\perp,\quad \xi=(\xi_1,\xi_2, 0)\in\mathbb{R}^3,\ \xi^\perp=(0, 0, \xi_3^\perp)\in\mathbb{R}^3,
\end{equation}
satisfying $\xi_1^2+\xi_2^2=(\xi_3^\perp)^2$. It is straightforwardly verified that $v(x)$ in \eqref{eq:expf}-\eqref{eq:expf2} defines a harmonic function. By taking $v(x)$ in \eqref{eq:orth1} to be the one defined in \eqref{eq:expf}, we have 
\begin{equation}\label{eq:ddf1}
\begin{split}
0=&\int_{\mathbb{R}^3} \varphi(x_1, x_2) e^{i\zeta\cdot x}\ dx=\int_{B_R} \varphi(x_1, x_2) e^{i\zeta\cdot x}\ dx\\
=&\int_{\mathbb{R}^2} e^{i\xi'\cdot x'}\varphi(x')\ dx'\cdot \int_{\{x_3; (x', x_3)\in B_R\}} e^{-\xi_3^\perp\cdot x_3}\ dx_3,
\end{split}
\end{equation}
where $x'=(x_1,x_2)\in \mathbb{R}^2$ and $\xi'=(\xi_1, \xi_2)\in\mathbb{R}^2$. In \eqref{eq:ddf1}, we have made use of the fact that $\varphi$ is supported in $B_R$. Clearly, from \eqref{eq:ddf1}, one readily has
\begin{equation}\label{eq:ddf2}
\int_{\mathbb{R}^2} e^{i\xi'\cdot x'}\varphi(x')\ dx'=0,
\end{equation}
which holds for any $\xi'\in\mathbb{R}^2$. The LHS of \eqref{eq:ddf2} is the Fourier transform of the function $\varphi$, and hence one easily infers from \eqref{eq:ddf2} that $\varphi=0$. 

The proof is complete. 
\end{proof}

\begin{proof}[Proof of Theorem~\ref{thm:25}]
Without loss of generality, we assume that $f$ is independent of the variable $x_3$ and write it as $f(x')$ for $x'=(x_1, x_2)\in\mathbb{R}^2$. By contradiction, we assume that there exists another TAT/PAT configuration $(\widetilde{f}(x'), \widetilde{c})$ supported in $\Omega$ with $\widetilde{c}$ being a constant such that 
\begin{equation}\label{eq:contr1n}
\Lambda_{f, c}(x, t)=\Lambda_{\widetilde{f},\widetilde{c}}(x, t)\quad\mbox{for}\quad (x, t)\in\partial\Omega\times\mathbb{R}_+. 
\end{equation} 

Since $c^{-2}f-\widetilde{c}^{-2}\widetilde{f}$ is independent of the variable $x_3$, we have by Theorem~\ref{thm:2} that
\begin{equation}\label{eq:f1n}
c^{-2}f=\widetilde{c}^{-2}\widetilde{f}. 
\end{equation}
Therefore,
\begin{equation}\label{eq:ff1n}
\int_{B_R} (c^{-2}f)(y)\cdot |x-y|\ dy=\int_{B_R} (\widetilde{c}^{-2}\widetilde{f})(y)\cdot |x-y|\ dy
\end{equation}
and
\begin{equation}\label{eq:ff2n}
\Delta^{-1}(c^{-2} f)=\Delta^{-1}(\widetilde{c}^{-2}\widetilde{f})
\end{equation}
Next, by \eqref{eq:orth2} in Theorem~\ref{thm:1}, together with the use of \eqref{eq:ff1n} and \eqref{eq:ff2n}, there clearly holds
\begin{equation}\label{eq:ff3n}
(c^{-2}-\widetilde{c}^{-2})\int_{B_R}\Delta^{-1}(c^{-2}f)(y)\cdot \Phi_0(x-y)\ dy=0,
\end{equation}
where $x\in\partial B_R$.  Then, by the assumption \eqref{eq:bp25} and Remark~\ref{rem:22}, it is directly shown that 
\begin{equation}\label{eq:ff4n}
\Delta^{-1}(c^{-2} f)(x)>0\quad\mbox{for\ } \ x\in \Omega. 
\end{equation}
Using \eqref{eq:ff4n}, one can easily verify that
\begin{equation}\label{eq:ff5n}
\int_{B_R} \Delta^{-1}(c^{-2} f)(x)\cdot \Phi_0(x-y)\ dy>0,\quad x\in\partial B_R,
\end{equation}
which together with \eqref{eq:ff2n} immediately implies that
\begin{equation}\label{eq:ff6n}
c=\widetilde{c}. 
\end{equation}
Finally, by \eqref{eq:ff6n} and \eqref{eq:f1n}, one has
\[
f=\widetilde{f}.
\]

The proof is complete. 

\end{proof}

\begin{rem}\label{rem:3}
We note that the assumption \eqref{eq:bp25} is only needed to guarantee that (cf. \eqref{eq:ff3n} and \eqref{eq:ff5n})
\begin{equation}\label{eq:ff7n}
\int_{B_R} \Delta^{-1}(c^{-2} f)(x)\cdot \Phi_0(x-y)\ dy\neq 0, \quad x\in B_R. 
\end{equation}
By using a completely similar argument to the proof of Theorem~\ref{thm:25}, one can show that if $f$ is nonpositive on $\Omega$ but strictly negative on an open subset. Then \eqref{eq:ff7n} also holds true. But as remarked in Remark~\ref{rem:22}, \eqref{eq:bp25} is a practical condition. The condition \eqref{eq:bp2} is more general, and this can be observed by simply taking $v\equiv 1$. 
\end{rem}

\begin{proof}[Proof of Theorem~\ref{thm:3}]
The proof follows from a similar argument to that of Theorem~\ref{thm:25}.  
Without loss of generality, we can assume that $f$ and $c_b$ are independent of the variable $x_3$, and write them as $c_b(x')$ and $f(x')$ for $x'=(x_1, x_2)\in\mathbb{R}^2$. By contradiction, we assume that there exists another TAT/PAT configuration $(\widetilde{f}(x'), \widetilde{c}(x'))$ supported in $\Omega$ with $\widetilde{c}(x')$ given as
\begin{equation}\label{eq:contr0}
\widetilde{c}^{-2}(x')=c_b^{-2}(x')+\widetilde{\gamma}^{-2}\chi_{\Sigma},
\end{equation}
where $\widetilde{\gamma}$ is a positive constant, such that
\begin{equation}\label{eq:contr1}
\Lambda_{f, c}(x, t)=\Lambda_{\widetilde{f},\widetilde{c}}(x, t)\quad\mbox{for}\quad (x, t)\in\partial\Omega\times\mathbb{R}_+. 
\end{equation}   

Noting that $c^{-2}f-\widetilde{c}^{-2}\widetilde{f}$ is independent of the variable $x_3$, we have by Theorem~\ref{thm:2} that
\begin{equation}\label{eq:f1}
c^{-2}f=\widetilde{c}^{-2}\widetilde{f}. 
\end{equation}
Similar to \eqref{eq:ff3n}, we have from \eqref{eq:orth2} in Theorem~\ref{thm:1} that
\begin{equation}\label{eq:ff8n}
\begin{split}
& \int_{B_R} (1-c^{-2}(y))\cdot \Delta^{-1}(c^{-2}f)(y)\cdot\Phi_0(x-y)\ dy\\
=& \int_{B_R} (1-\widetilde{c}^{-2}(y))\cdot \Delta^{-1}(\widetilde{c}^{-2}\widetilde{f})(y)\cdot\Phi_0(x-y)\ dy
\end{split}
\end{equation}
for any $x\in B_R$. Using a completely similar argument in proving \eqref{eq:orth1} in the proof of Theorem~\ref{thm:1}, one can show from \eqref{eq:ff8n} that
\begin{equation}\label{eq:ff9n}
\begin{split}
& \int_{B_R} (1-c^{-2}(y))\cdot \Delta^{-1}(c^{-2}f)(y)\cdot v(y)\ dy\\
=& \int_{B_R} (1-\widetilde{c}^{-2}(y))\cdot \Delta^{-1}(\widetilde{c}^{-2}\widetilde{f})(y)\cdot v(y)\ dy
\end{split}
\end{equation}
for any harmonic function $v$. 

Next, by combining \eqref{eq:bp}, \eqref{eq:contr0}, \eqref{eq:f1} and \eqref{eq:ff9n}, we have
\begin{equation}\label{eq:ff3}
(\gamma^{-2}-\widetilde{\gamma}^{-2})\int_{\Sigma}\Delta^{-1}(c^{-2} f)(x)\cdot v(x)\ dx=0
\end{equation}
for any harmonic function $v(x)$. Noting \eqref{eq:bp2}, we immediately have from \eqref{eq:ff3} that 
\[
\gamma=\widetilde{\gamma},
\]
which in turn, together with \eqref{eq:f1}, implies that
\[
f=\widetilde{f}.
\]

The proof is complete. 
\end{proof}

\section*{Acknowledgement}

The work of H. Liu was supported by the FRG and start-up grants of Hong Kong Baptist University, and the NSF grant of China, No.\,11371115. The work of G. Uhlmann was supported by NSF.

\end{document}